\documentclass[a4paper,12pt]{article}
\usepackage{index}
\usepackage[arrow, matrix, curve]{xy}
\usepackage{graphicx}
\usepackage[utf8]{inputenc}
\usepackage{longtable}
\usepackage{psboxit}
\usepackage{vmargin,hhline,pifont}
\usepackage{enumerate,fancyhdr}
\usepackage{amsmath}
\usepackage{amsthm}
\usepackage{amsfonts}
\usepackage{amssymb}
\usepackage{chapterbib}
\usepackage{marvosym}
\usepackage{verbatim}

\newtheorem{theorem}{Theorem}[section]

\newtheorem{lemma}[theorem]{Lemma}

\theoremstyle{definition}
\newtheorem{definition}[theorem]{Definition}
\newtheorem{example}[theorem]{Example}

\newtheorem{corollary}[theorem]{Corollary}
\theoremstyle{remark}
\newtheorem{remark}[theorem]{Remark}

\newcommand{\ndash}{\nobreakdash-\hspace{0pt}}

\begin{document}

\title{Tame Fr\'echet submanifolds}
\author{Walter Freyn \footnote{Fachbereich Mathematik, TU Darmstadt,
Schlossgartenstrasse 7,64289 Darmstadt, Germany, walter.freyn@math.tu-darmstadt.de}
}

\maketitle

\abstract{We introduce the new class of submanifolds of co-Banach type in tame Fr\'echet manifolds and construct tame Fr\'echet submanifolds as inverse images of regular values of certain tame maps. Our method furnishes an easy way to construct tame Fr\'echet manifolds. The results presented are key ingredients in the construction of affine Kac-Moody symmetric spaces; they have also important applications in the study of isoparametric submanifolds in tame Fr\'echet spaces.}

\section{Introduction}

In this paper we introduce and investigate a new class of submanifolds in tame Fr\'echet manifolds, which we call \emph{submanifolds of co-Banach type}. Our interest in this class of submanifolds arose from its natural emergence in Kac-Moody geometry~\cite{Freyn10a}.

The study of finite dimensional submanifolds in finite dimensional ambient spaces is a well-developed field displaying connections to manifold parts of mathematics and physics. In contrast, the theory of submanifolds in infinite dimensional ambient spaces is much less developed.  For infinite dimensional submanifolds the most important challenge is the appropriate choice of the functional analytic framework imposed on the submanifolds and their ambient spaces. While the study of certain classes of submanifolds in Hilbert spaces is a by now classical topic~\cite{Kuiper71, PalaisTerng88, Terng89, HeintzeGorodski12}, the investigation of submanifolds in ambient spaces which are neither Banach- nor Hilbert spaces has not gained much attention. Nevertheless,  as in important examples the functional analytic nature of the ambient space is determined by geometric or algebraic requirements, submanifolds in ambient spaces, which are not Banach spaces, occur naturally.  

The author's interest in this topic arose from constructions in affine Kac-Moody geometry. 
Thus let us sketch in some lines its setting:  Affine Kac-Moody geometry is the infinite dimensional differential geometry of affine Kac-Moody groups. Affine Kac-Moody groups can be thought of as the closest infinite dimensional generalization of semisimple Lie groups~\cite{Kac90, Moody95}; they can be described as certain torus extensions $\widehat{L}(G, \sigma)$ of (possibly twisted) loop groups $L(G, \sigma)$. Correspondingly affine Kac-Moody algebras are $2$\ndash dimensional extensions of (possibly twisted) loop algebras $L(\mathfrak{g}, \sigma)$. Here $G$ denotes a compact or complex simple Lie group, $\mathfrak{g}$ its Lie algebra and $\sigma\in \textrm{Aut}(G)$ a diagram automorphism, defining the ``twist''. Imposing suitable regularity assumptions on the loops, one obtains families of completions of the minimal (=algebraic) affine Kac-Moody groups.  This minimal algebraic loop group consists of Laurent polynomials. Following Jacques Tits, completions defined by imposing regularity conditions on the loops are called ``analytic completions'' in contrast to the more abstract formal completion. Various analytic completions and objects closely related to them play an important role in different branches of mathematics and physics, especially quantum field theory, integrable systems and differential geometry. In most instances their use is motivated by the requirement to employ functional analytic methods or by the need to work with manifolds and Lie groups \cite{PressleySegal86}, \cite{Guest97}, \cite{SegalWilson85}, \cite{Tsvelik95}, \cite{Popescu05}, \cite{PalaisTerng88}, \cite{Kobayashi11}, \cite{Khesin09}, \cite{HPTT}, \cite{HeintzeLiu}, \cite{Heintze06} and references therein. The study of functional analytic settings suitable for infinite dimensional differential geometry was performed for example in~\cite{Kriegl97, Glockner06, Neeb06, Neeb07, Glockner07}. See also~\cite{BertramGlocknerNeeb04, Bertram08, AnNeeb09}.

It turns out that for the definition of symmetric spaces associated to affine Kac-Moody groups, one is constrained to use the completion of affine Kac-Moody groups (resp. algebras) as groups (resp. algebras) of holomorphic maps of $\mathbb{C}^*$ into complex simple Lie groups (resp. Lie algebras)~\cite{Freyn09}. These spaces of holomorphic maps cannot be made into Hilbert or Banach spaces; nevertheless they carry a natural structure as Fr\'echet spaces. Unfortunately Fr\'echet spaces have in general unsatisfactory analytic properties: the main impediment to the development of a successful theory is the failure of an inverse function theorem. To circumvent this problem, various approaches were proposed: in most cases these approaches proceed by picking a distinguished subclass of Fr\'echet spaces, that show better properties and allow for the proof of an inverse function theorem. Following these lines of thought and abstracting arguments developed by John Nash in the proof of his famous embedding theorem~\cite{Nash56}, the class of  \emph{tame Fr\'echet spaces} was defined by Richard Hamilton in~\cite{Hamilton82}. For tame Fr\'echet spaces a convenient inverse function theorem exists. Some applications of tame Fr\'echet spaces to Lie groups and Lie algebras were investigated in~\cite{Payne89, Payne91}.

 The aim of this paper is to establish certain foundational results about tame Fr\'echet submanifolds of co-Banach type and the more restricted tame Fr\'echet submanifolds  co-finite type: The functional analytic key observation is, to establish that submanifolds of co-Banach type inherit a natural tame Fr\'echet structure from their ambient tame Fr\'echet manifolds resp. tame Fr\'echet spaces.  This result is used as a functional analytic key observation in a research program by the author, to construct Kac-Moody symmetric spaces and more generally to develop Kac-Moody geometry~\cite{Freyn10a, Freyn12d, Freyn12e}. It is also used in the theory of pseudo-Riemannian affine Kac-Moody symmetric spaces~\cite{Freyn12a}.   
Among other things we prove an implicit function theorem and show, that inverse images of regular values of certain maps are tame Fr\'echet submanifolds. This result is used for example in~\cite{Freyn12d} in the proof, that affine Kac-Moody groups are tame Fr\'echet manifolds. Remark that our notion of a co-Banach submanifold is stronger than other notions of submanifolds in the Fr\'echet setting, as we require the tangential space of the submanifold to be complemented by a Banach space.

Let us describe the content of the paper in a little more detail. 

Recall first the well-known finite dimensional blueprint. In this finite dimensional situation there are many ways known, to construct submanifolds $M^{k}$ in some finite dimensional vector space $V\cong \mathbb{R}^{n}$ such that $n>k$.  Furthermore any submanifold carries automatically the structure of an abstract manifold. One of the most commonly used criteria to establish, that a subset $M^{k}\subset V^n$ is a submanifold is the following: Let $f:\mathbb{R}^{n}\longrightarrow \mathbb{R}^{n-k}$ be a smooth map and $p\in \mathbb{R}^{n-k}$ a regular point (that is, a point $p$ such that for any point $y$ in the inverse image $y\in f^{-1}(p)$ the differential 
$df_{y}:T_{y}\mathbb{R}^{n}\longrightarrow T_{p}\mathbb{R}^{n-k}$ is surjective). Then an important theorem in differential topology states, that the inverse image 
$$\left\{y\in \mathbb{R}^{n}| f(y)=p \right\}$$
is a submanifold~\cite{BroeckerJaenich82}. A typical application is the proof that orthogonal groups $O(n)$ are manifolds, by viewing them as subsets in the set of $n\times n$-matrices and hence as a subset in some $n\times n$ dimensional vector space. Then one shows easily, that the identity matrix is a regular value of the map $A\mapsto AA^*$ whose inverse image is exactly the orthogonal group.

In this article we generalize this method to tame Fr\'echet spaces. To this end we need to provide new versions of all ingredients in this proof. Hence we need to 

\begin{itemize}
\item prove an implicit function theorem, adapted to our situation,
\item define a new notion of a \emph{regular value},
\item establish the tame Fr\'echet submanifold structure.
\end{itemize}

We start in section~\ref{sect:tame_spaces} to review basic ideas and definitions about tame Fr\'echet spaces as described in~\cite{Hamilton82}; then we establish a characterization of tame Fr\'echet space as subspaces of certain spaces of holomorphic maps. In section~\ref{sect_submanifolds_finite_type} we introduce the new notions of tame Fr\'echet manifolds of co-Banach and co-finite type. In section~\ref{tame_Frechet_manifold} we describe tame Fr\'echet manifolds. In section~\ref{sect:implicit_function_theorem} we prove our version of the implicit function theorem, which is based on the Nash-Moser inverse function theorem. In the final section~\ref{sect:construction_submanifolds} we use this result, to construct some tame Fr\'echet submanifolds of co-finite type. This paper is part of a research program by the author devoted to present a complete theory of affine Kac-Moody symmetric spaces. The results of this paper form a part of the technical core, to provide the functional analytic basics for the theory. A special case of the main result of this paper was established in the author's thesis~\cite{Freyn09}.

My special thanks go to the unknown referee for numerous remarks and comments, which led to important improvements of the article.

\section{Tame Fr\'echet spaces}
\label{sect:tame_spaces}
\subsection{Fr\'echet spaces}

\noindent This introductory section collects some standard results about Fr\'echet spaces, Fr\'echet manifolds and Fr\'echet Lie groups. Further details or omitted proofs can be found in Hamilton's article~\cite{Hamilton82}.

\begin{definition}[Fr\'echet space]
A Fr\'echet vector space is a locally convex topological vector space which is complete, Hausdorff and metrizable.
\end{definition}

\begin{lemma}[Metrizable topology]
A vector topology on a vector space is metrizable iff it can be defined by a countable collection of seminorms.
\end{lemma}

\noindent Let us give some typical examples:

\begin{example}[Fr\'echet spaces]~
\label{frechetexamples}
\begin{enumerate}
\item Every Banach space is a Fr\'echet space. The countable collection of norms contains just one element.
\item Let $\textrm{Hol}(\mathbb{C}, \mathbb{C})$ denote the space of holomorphic functions $f: \mathbb{C} \longrightarrow \mathbb{C}$. Let $K_n$ be a collection of simply connected compact sets in $\mathbb C$, such that $\bigcup K_n=\mathbb{C}$. Let $\|f\|_n:= \displaystyle\sup_{z\in K_n} |f(z)|$. Then $\textrm{Hol}(\mathbb {C}, \mathbb {C}; \|\hspace{3pt}\|_n)$ is a Fr\'echet space. 
\end{enumerate}
\end{example}
 
\begin{definition}
 A Fr\'echet manifold is a (possibly infinite dimensional) manifold with charts in a Fr\'echet space such that the chart transition functions are smooth. 
\end{definition}

While it is possible to define Fr\'echet manifolds in this way, there are three strong impediments to the development of analysis and geometry of those spaces:
\begin{enumerate}
\item In general there is no inverse function theorem for smooth maps between Fr\'echet spaces. For counterexamples and examples pointing out features special to Fr\'echet spaces, see~\cite{Hamilton82}.
\item In general the dual space of a Fr\'echet space is not a Fr\'echet space. 
\item there is no usefull topological structure on the general linear group of a Fr\'echet space known~\cite{Dodson12, DodsonGalanis04}.
\end{enumerate}

\noindent To tackle these problems successfully is difficult:  
\begin{example}
We take as example the space of holomorphic functions $\textrm{Hol}(\mathbb{C}, \mathbb{C})$. The domain $\mathbb{C}$ can be interpreted as a direct limit of the sets $K_n:=B(0;n)$ of balls of radius $n$ around $0$ with respect to inclusion;  following the construction in example~\ref{frechetexamples}, example 2, we can define a sequence of norms $\|f\|_{n,n\in \mathbb{N}}$ on $\textrm{Hol}(\mathbb{C}, \mathbb{C})$ by $\|f\|_n:=\sup_{z\in K_n}|f(z)|$. The space of functions on a direct limit is an inverse limit; $\textrm{Hol}(\mathbb{C}, \mathbb{C})$ coincides thus with the inverse limit of a sequence of function spaces $\textrm{Hol}(K_n, \mathbb{C})$, where 
\begin{displaymath}\textrm{Hol}(K_n, \mathbb{C}):=\bigcup_{K_n\subset U_{n}^m}\{\textrm{Hol}(U_{n}^m, \mathbb{C})\}\, ,\end{displaymath}

and $U_n^m$ runs through all open sets containing $K_n$.
 By a choice of appropriate norms on the spaces $\textrm{Hol}(K_n, \mathbb{C})$, one can give them structures as subspaces of Bergmann-(or Hardy-) spaces. See for example~\cite{HKZ00} and \cite{Duren00}. Hence $\textrm{Hol}(\mathbb{C}, \mathbb{C})$ can be interpreted as inverse limit of Hilbert spaces. By categorical duality, the dual space of an inverse limit is a direct limit and vice versa. Hence the dual space of $\textrm{Hol}(\mathbb{C}, \mathbb{C})$ is as a dual space not a Fr\'echet space. See also~\cite{Zarnadze92} and references therein.
 \end{example} 
 
Similar arguments hold for reflexive Fr\'echet spaces. A detailed study may be found in~\cite{Schaefer80, Kriegl97}.

The solution to the first problem is based on a more refined control of the inverse limits. This is achieved by defining carefully chosen comparison conditions between the various seminorms. Using this structure, the inverse function theorems on the Hilbert-(resp.\ Banach-)spaces in the inverse limit sequence piece together to give an inverse function theorem on the limit space; this is the famous Nash-Moser inverse function theorem. In the next sections we will formalize those concepts. 

To bypass the third obstacle, we use in this paper only general linear groups of Banach spaces. This is where the condition of complementary Banach spaces in the definition of co-Banach submanifolds in definition~\ref{TamesubmanifoldofBanachtype}  is necessary.

We have to note that there are other ways proposed of getting control over Fr\'echet spaces. A recent example is the concept of bounded Fr\'echet geometry developed by Olaf M\"uller~\cite{Muller06}.

\subsection{Tame Fr\'echet spaces}

A central challenge in the development of any advanced structure theory of Fr\'echet spaces is a better control of the set of seminorms. For Fr\'echet spaces $F$ and $G$ and a map $\varphi: F\longrightarrow G$ this will be done by imposing estimates similar in spirit to the concept of quasi isometries relating the sequences of norms $\|\varphi(f)\|_n$ and $\|f\|_m$. Following~\cite{Hamilton82} \emph{tame Fr\'echet spaces} are defined as Fr\'echet spaces that are \emph{tame equivalent} to some model space of holomorphic functions.

The prerequisite for estimating norms under maps between Fr\'echet spaces are estimates of the norms on the Fr\'echet space itself. This is done by a grading:

\begin{definition}[grading]
Let $F$ be a Fr\'echet space. A grading on $F$ is a sequence of seminorms $\{\|\hspace{3pt}\|_{n}, n\in \mathbb N_0\}$ that define the topology and satisfy
$$\|f \|_0\leq \|f\|_1 \leq \|f\|_2 \leq\| f \|_3 \leq \dots \,.$$
\end{definition}

\begin{lemma}[Constructions of graded Fr\'echet spaces]~
\begin{enumerate}
	\item A closed subspace of a graded Fr\'echet space is a graded Fr\'echet space.
	\item Direct sums of graded Fr\'echet spaces are graded Fr\'echet spaces.
\end{enumerate}
\end{lemma}

\noindent EAny Fr\'echet space admits a grading. Let  $(F, \|\hspace{3pt}\|_{n, n\in \mathbb{N}})$ be a Fr\'echet space. Then 
\begin{displaymath}
\left(\widetilde{F}, \widetilde{\|\hspace{3pt}\|}_{n, n\in \mathbb{N}}\right)
\end{displaymath} 
such that $\widetilde{F}:=F$ as a vector space and $\widetilde{\|\hspace{3pt} \|}_{n}:=\displaystyle\sum_{i=1}^n \|\hspace{3pt}\|_i$ is a graded Fr\'echet space. As Fr\'echet spaces $F$ and $\widetilde{F}$ are isomorphic. In consequence the existence of a grading is not a property a Fr\'echet space may satisfy or may not satisfy. A grading is an additional structure, defined on the Fr\'echet space, which is of geometric nature. A useful notion of equivalence on the space of gradings of a given Fr\'echet space is the one of \emph{tame equivalence}:

\begin{definition}[Tame equivalence of gradings]
Let $F$ be a graded Fr\'echet space, $r,b \in \mathbb{N}$ and $C(n), n\in \mathbb{N}$ a sequence with values in $\mathbb{R}^+$. The two gradings  $\{\|\hspace{3pt}\|_n\}$ and $\{\widetilde{\|\hspace{3pt}\|}_n\}$ are called $\left(r,b,C(n)\right)$-equivalent iff 
\begin{equation*}
 \|f\|_n \leq C(n) \widetilde{\|f\|}_{n+r} \text{ and }  \widetilde{\|f\|}_n \leq C(n)\|f\|_{n+r} \text{ for all } n\geq b
\,.
\end{equation*}
They are called tame equivalent iff they are $(r,b,C(n))$-equivalent for some choice $(r,b,C(n))$.
\end{definition}

Remark that in general on a Fr\'echet space, there are gradings which are not tame equivalent. Hence the tame structure is finer than the topological structure.

\noindent The following example is basic:

\begin{example}
\label{example:two equivalent norms}
Let $B$ be a Banach space with norm $\| \hspace{3pt} \|_B$. Recall that a sequence $(f_k)_{k\in \mathbb{N}}$ of elements of $B$ is called exponentially decreasing iff $\sum_k \|f_k\|e^{kn}<\infty$ for all $n\in \mathbb{N}$. Denote by $\Sigma(B)$ the space of all exponentially decreasing sequences $(f_k)_{k\in \mathbb{N}}$ on $B$.
On this space the following sequences of norms define gradings:

\begin{align}
\|f\|_{l_1^n} &:= \sum_{k=0}^{\infty}e^{nk} \|f_k\|_B\,,\\
\|f\|_{l_{\infty}^n}&:= \sup_{k\in \mathbb N_0} e^{nk}\|f_k\|_B\,.
\end{align}
\end{example}

\begin{lemma}
\label{lemma:two equivalent norms}

On the space $\Sigma(B)$ the two gradings $\|f\|_{l_1^n}$ and $\|f\|_{l_{\infty}^n}$ are tame equivalent.
\end{lemma}

\noindent For the proof see~\cite{Hamilton82}.

\begin{example}
The space of exponentially decreasing sequences of elements in $B=\mathbb C^2$ equipped with the Euclidean norm and the space of exponentially decreasing sequences of elements in $B=\mathbb C^2$ together with the supremum-norm $\|(c_1, c_1')\|_B :=\sup(|c_1|,|c_1'|)$ are tame Fr\'echet spaces.
\end{example}

Let us give an intuitive characterization of tame spaces which seems not be in the literature:

\begin{lemma}
\label{lemma:sigma_B_is_holomorphic}
Let $B$ be a complex Banach space. The space $\Sigma(B)$ of exponentially decreasing sequences in $B$ is isomorphic to the space $\textrm{Hol}(\mathbb{C},B)$ of $B$\ndash valued holomorphic functions.
\end{lemma}

\begin{proof}
We prove the inclusions $\Sigma(B)\subset \textrm{Hol}(\mathbb{C},B)$ and $\textrm{Hol}(\mathbb{C},B)\subset \Sigma(B)$:
Let first $f\in \textrm{Hol}(\mathbb{C},B)$ with a series expansion:
\begin{displaymath}
f(z)=\sum_{k=0}^{\infty}f_k z^k.
\end{displaymath} 
The coefficients $f_k$ are elements of $B$ and we have to show that the sequence $(f_k)$ is an exponentially decreasing sequence. In view of lemma~\ref{lemma:two equivalent norms} we have to establish, that the
 $l_{\infty}$\ndash norm is bounded.
To deduce some estimate for $f_k$ we need three ingredients:
\begin{enumerate}
 \item As $f(z)$ is an entire function the Taylor series expansion $f(z)=\sum_{n=0}^{\infty}f_k z^k$ converges for all $z\in \mathbb{C}$ absolutely - hence 
 $$\sup_{z\in B(0,e^n)}|f(z)|<\infty\qquad \textrm{for all} \quad n\,.$$
\item Differentiation of $f$ yields the identity $f_k=\frac{1}{k!}f^{(k)}(0)$.
\item The Cauchy inequality~\cite{Berenstein91}, 2.1.20: Let $f$ be holomorphic on $B(0,r)$. Then
\begin{displaymath}
\left|f^{(k)}(0)\right|\leq k!\frac{\sup_{z\in B(0,r)}|f(z)|}{r^k}\,.
\end{displaymath}
\end{enumerate}
Using these ingredients we get for $n\in \mathbb{N}$

\begin{align*}
 \sup_{k}|f_k|e^{nk}&=\sup_{k}\left|\frac{f^{(k)}(0)}{k!} \right|e^{kn}\leq\\
&\leq\sup_k\frac{e^{kn}}{k!}\left|k!\frac{\sup_{z\in B(0,e^n)}|f(z)|}{e^{nk}}\right|=\\
&=\sup_{z\in B(0,e^n)}|f(z)|\leq \infty\,.
\end{align*}
Hence 
\begin{displaymath}
 \|f\|_{l_{\infty}^{n}}<\infty\, .
\end{displaymath}

Conversely for any exponentially decreasing sequence $(f_n)_{n\in \mathbb{N}}$, $f_n\in B$ we claim, that 
\begin{displaymath}
f(z):=\sum_{n=0}^{\infty}f_n z^n\, .
\end{displaymath} 
defines a $B$-valued entire holomorphic function. To see this we use for $U\subset B(0,e^n)$  the estimate
\begin{displaymath}
 |f(z)|=\left|\sum_k f_k z^k\right|\leq \sum_k |f_k||z^k|\leq \sum |f_k|e^{kn}<\infty\quad\textrm{for}\ z\in U\,.
\end{displaymath}
\end{proof}

\begin{corollary}
\label{cor:sigma_B_is_holomorphic}
Let $B$ be a real Banach space and $B_{\mathbb{C}}$ its complexification. Then $\Sigma(B)$ consists of all $B_{\mathbb{C}}$\ndash valued holomorphic functions on $\mathbb{C}$ such that 
\begin{displaymath}
 \overline{f(z)}=f(\overline{z})\, .
\end{displaymath}
  
\end{corollary}

\begin{example}
 The prototypical example is the choice $B:=\mathbb{C}$. Then we find
\begin{displaymath}
 \Sigma(\mathbb{C}):=\left\{(a_k)_{k\in \mathbb{N}_0}\in \mathbb{C}\mid\sum_k |a_k|e^{kn}<\infty\quad \forall n\in \mathbb{N}\right\}\,.
\end{displaymath}
 Hence $\Sigma(\mathbb{C})=\textrm{Hol}(\mathbb{C},\mathbb{C})$.
\end{example}

Let $F$, $G$, $G_1$ and $G_2$  denote graded Fr\'echet spaces.

\begin{definition}[Tame linear map]
A linear map $\varphi: F\longrightarrow G$ is called $(r,b,C(n))$-tame if it satisfies the inequality
$$\|\varphi(f)\|_n \leq C(n)\|f\|_{n+r}\, \quad \textrm{for all }n\geq b.$$
$\varphi$ is called tame iff it is $(r,b,C(n))$-tame for some $(r,b, C(n))$.
\end{definition}

Remark that tameness of a map does depend on the choice of the set of norms. Two different sets of norms generating the same topology may define different tame structures~\cite{Hamilton82}.

\begin{definition}[Tame isomorphism]
A map $\varphi:F\longrightarrow G$ is called a tame isomorphism iff it is a linear isomorphism and $\varphi$ and $\varphi^{-1}$ are tame maps.
\end{definition}

\begin{definition}[Tame direct summand]
$F$ is a tame direct summand of $G$ iff there exist tame linear maps $\varphi: F\longrightarrow G$ and $\psi: G \longrightarrow F$ such that $\psi \circ \varphi: F \longrightarrow F$ is the identity.
\end{definition}

Hence tame direct summands are defined by one-sided inverses. A \emph{tame Fr\'echet space} is now defined as a tame direct summand of a model space.

\begin{definition}[Tame Fr\'echet space]
$F$ is a tame Fr\'echet space iff there is a Banach space $B$ such that $F$ is a tame direct summand of $\Sigma(B)$.
\end{definition}

In view of \ref{lemma:sigma_B_is_holomorphic} and corollary \ref{cor:sigma_B_is_holomorphic} tame Fr\'echet spaces are essentially spaces of holomorphic maps. Let us remark the following result of~\cite{Hamilton82}:

\begin{lemma}[Constructions of tame Fr\'echet spaces]~
\label{constructionoftamespaces}
\begin{enumerate}
	\item A tame direct summand of a tame Fr\'echet space is tame.
	\item A Cartesian product of two tame Fr\'echet spaces is tame.
\end{enumerate}
\end{lemma}

For the omitted proofs and additional examples see~\cite{Hamilton82}.
\noindent Let us review example \ref{frechetexamples}. 

\begin{example}
\begin{enumerate}
\item Banach spaces, are trivially tame Fr\'echet spaces with $B\hookrightarrow \Sigma(B), b\mapsto (b,0,0,\dots)$ which is trivially exponentially decreasing .
\item For the space of holomorphic functions on $\mathbb{C}$ tameness depends on the sequence of norms and hence on choice of the compact subsets $K_n$.
\begin{itemize}
\item The sequence of norms $\|f\|^K_n:= \displaystyle\sup_{z\in K_n} |f(z)|$ is graded iff $K_{n}\subset K_{n+1}$. 
\item Let $B(0;n):=\{x\in \mathbb{C}|\|x\|\leq n\}$ and $$\|f\|_n:=\sup_{z\in B(0;n)}|f(z)|\,.$$ The tame Fr\'echet space $(Hol(\mathbb{C}, \mathbb{C}; \|\phantom{z}\|_n)$ is $(r,b,C(n)\equiv 1)$-tame equivalent to the Fr\'echet space $(\textrm{Hol}(\mathbb{C}, \mathbb{C}; \|\phantom{z}\|^K_n)$ iff there are some integers $b,k\geq 0$ such that $B(0;n)\subset K_{n+k}\subset B(0,n+2k)$ for all $n\geq b$.
\end{itemize}
\end{enumerate} 
\end{example}

\noindent We now describe some nonlinear concepts in the category of tame Fr\'echet spaces: we start by extending the concept of tame maps from linear to non-linear maps:

\begin{definition}
A nonlinear map $\Phi: U\subset F \longrightarrow G$ is called $(r, b, C(n))$-tame iff it satisfies the inequality 
\begin{displaymath}
\|\Phi(f)\|_{n}\leq C(n)(1+\| f\|_{n+r})\,\forall n>b\,.
\end{displaymath}
$\Phi$ is called tame iff it is $(r,b,C(n))$-tame for some $(r,b, C(n))$.
\end{definition}

\begin{example}
\label{Example:Banach_tame_maps}
Let $F$ and $G$ be Banach spaces (hence the collection of norms consists of one norm) and $\Phi:F\longrightarrow G$ is a $(r_1,b_1,C_1)$\ndash tame isomorphism with a $(r_2, b_2,C_2)$\ndash tame inverse $\Phi^{-1}$. If $b_1\geq 2$ and $b_2\geq 2$ the condition on the norms vanishes. If $b_1=r_1=b_2=r_2=0$ we get 
\begin{displaymath}
\|\Phi(f)\|\leq C_1(1+\| f\|)\,\quad\textrm{and}\quad \|\Phi^{-1}(g)\|\leq C_2(1+\| g\|)\, .
\end{displaymath}
After some elementary manipulations we find
\begin{displaymath}
 \frac{1}{C_2}\|f\|-1\leq \|\Phi(f)\|\leq C_1(1+\|f\|)\,.
\end{displaymath}
This is the definition of a quasi-isometry~\cite{Burago01}. Hence for Banach spaces tame isomorphisms are exactly quasi-isometries.
\end{example}

This example highlights the real nature of the tameness condition: The tameness condition is a notion of quasi-isometries, adapted to the functional analytic setting of Fr\'echet spaces. Different tame structures on the same Fr\'echet spaces can be viewed as a generalization of the notion of different quasi-isometry classes on the same topological space.

\begin{lemma}[Construction of tame maps]~
\label{constructionoftamemaps}
Let $G_{i, i=1,\dots k}$ be tame Fr\'echet spaces and define $G=\bigotimes_{i=1}^k G_i$.
\begin{enumerate}
\item Let $\Phi: U\subset F \longrightarrow G$ be a tame map. Define the projections $\pi_i:G \longrightarrow G_i, i=1,\dots, n$. The maps
\begin{displaymath}
\Phi_i:=\pi_i \circ \Phi: U\longrightarrow G_i
\end{displaymath}
are tame as well.
\item Let $\Phi_i: U \subset F\longrightarrow G_{i, i\in 1,\dots, k}$ be $(r_i, b_i, C_i(n))$-tame maps. Then the map
\begin{displaymath}
\Phi:=(\Phi_1,\dots, \Phi_k): U \longrightarrow G
\end{displaymath}
 $$\textrm{is}\quad\left(\max(r_i),\max(b_i), \sum_{i=1}^k C_i(n)\right)\quad\textrm{-tame.}$$
\end{enumerate}
\end{lemma}

\begin{proof}~
\begin{enumerate}
\item Projections onto a direct factor are $\left(0,0,(1)_{n\in \mathbb{N}}\right)$-tame. The composition of tame maps is tame. Thus $\Phi_{i}$ is tame.
\item By induction on $k$ it is enough to show this in the special case $k=2$. 
Let $f\in U\subset F$. \begin{align*}
\|\Phi(f)\|_n &= \|\Phi(f)\|^1_n+\|\Phi(f)\|^2_n \leq\\
              &\leq C_1(n)(1+\| f\|_{n+r_1})+ C_2(n)(1+\|f\|_{n+r_2}) \leq\\
							&\leq C_1(n)(1+\| f\|_{n+\max(r_1, r_2)})+ C_2(n)(1+\|f\|_{n+\max(r_1, r_2)})= \\
							&\leq (C_1(n)+C_2(n))(1+\| f\|_{n+\max(r_1, r_2)})
\end{align*}
for all $n\geq \max(b_1, b_2)$.
\end{enumerate}
\end{proof}

\begin{remark}
The conclusion of the second part of Lemma~\ref{constructionoftamemaps} is correct only for finite products.
\end{remark}

\subsection{The Nash-Moser inverse function theorem}

\noindent The ultimate reason for the introduction of the category of tame Fr\'echet spaces and tame maps is the Nash-Moser inverse function theorem. 
We quote the version of~\cite{Hamilton82}, theorem 1.1.1.:

\begin{theorem}[Nash-Moser inverse function theorem]
Let $F$ and $G$ be tame (Fr\'echet) spaces and $\Phi: U\subseteq F \longrightarrow G$ a smooth tame map. Suppose that the equation for the derivative $D\Phi(f)h=k$ has a unique solution $h=V\Phi(f)k$ for all $f\in U$ and all $k$ and that the family of inverses $V\Phi:U\times G \longrightarrow F$ is a smooth tame map. Then $\Phi$ is locally invertible, and each local inverse $\Phi^{-1}$ is a smooth tame map.
\end{theorem}

\noindent A description of this theorem, a proof and some of its applications may be found in the article~\cite{Hamilton82}.

The crucial difference to the well-known inverse function theorem for Banach spaces is the following: In the classical inverse function theorem for Banach spaces it is enough to assume that the differential is invertible in a single point, to deduce invertibility of the function in a whole open neighborhood $U\ni p$. Hence invertibility in one point is enough to get nonlocal information. The important additional assumption in the situation of Fr\'echet spaces is that the invertibility of the differential has to be assumed not only in a single point $p$ but in a small neighborhood $U$ around $p$. 

This additional condition is necessary~\cite{Hamilton82} because in contrast to the Banach space situation it is not true in the Fr\'echet space case that the existence of an invertible differential in one point leads to invertibility in a neighborhood. Remark that this is a condition, that changes significantly the behavior of functions with respect to invertibility.

Let us quote the following result~\cite{Hamilton82}, theorem 3.1.1.\ characterizing the family of smooth tame inverses.

\begin{theorem}
 Let $L: (U \subseteq F) \times H \longrightarrow K$ be a smooth tame family of linear
maps. Suppose that the equation 
\begin{displaymath}
L(f)h = k
\end{displaymath} has a unique solution $h$ for all $f$ and
$k$ and that the family of inverses $V(f)k = h$ is continuous and tame as a map from $K$ to $H$.
Then $V$ is also a smooth tame map to $H$.
\begin{displaymath}
V:(U\subseteq F)\times K\longrightarrow H.
\end{displaymath}
\end{theorem}

\subsection{Tame Fr\'echet manifolds}
\label{tame_Frechet_manifold}

Let us now introduce the notion of a tame Fr\'echet manifold~\cite{Hamilton82}.

\begin{definition}[Tame Fr\'echet manifold]
A tame Fr\'echet  manifold is a Fr\'echet manifold with charts in a tame Fr\'echet space
such that the chart transition functions are smooth tame maps.
\end{definition}

\begin{example}
\label{example:BanachmanifoldistameFrechet}
In view of example~\ref{Example:Banach_tame_maps}, a Banach manifold is a tame Fr\'echet manifold iff it admits an atlas such that the chart transition functions are quasi-isometries.
\end{example}

\begin{definition}
Let $M$ and $N$ be two tame Fr\'echet manifolds modeled on $F$ resp. $G$. 
\begin{itemize}
\item A map $f:F\longrightarrow N$ is tame if for any chart $\psi: N\rightarrow U_{\psi}\subset G$ the concatenation $$f\circ \psi : F\longrightarrow U_{\psi}$$is tame whenever it is defined.
\item A map $f: M\longrightarrow N$ is tame iff for every pair of charts $\psi_i:V_i \subset N \longrightarrow V_i'$ and $\varphi_j: U_i \subset M \longrightarrow U_i'$, the map $$\psi_i \circ f \circ \varphi_j^{-1}:U_i'\longrightarrow V_i'$$ is tame whenever it is defined. 
\end{itemize}
\end{definition}

\section{Submanifolds of co-Banach and co-finite type}
\label{sect_submanifolds_finite_type}

In this section we introduce the new notion of \emph{tame Fr\'echet submanifolds of co-Banach type} and \emph{tame Fr\'echet submanifolds of co-finite type}. In a first subsection we start with a review of common finiteness conditions before we introduce the new classes of tame Fr\'echet submanifolds in the second subsection.

\subsection{Finiteness conditions}

An incautious generalization of notions of finite differential geometry to the infinite dimensional situation usually fails. To preserve the validity of nontrivial results in most instances finiteness conditions are required.
These finiteness conditions tend to be either of functional analytic nature (i.e. convergence of certain infinite sums) or of geometric nature (finiteness requirements on certain objects). Typical examples (which are also the most important ones for the applications in Kac-Moody geometry, which we have in mind) are
\begin{itemize}
 \item Fredholm conditions and compactness conditions,
 \item finite codimensions.
\end{itemize}

These conditions were introduced in various settings for different reasons. For example in her theory of isoparametric submanifolds in Hilbert spaces  Chuu-Lian Terng~\cite{Terng89} requires submanifolds to fulfill a Fredholm condition and a finite codimension condition. These conditions were introduced with the aim, to define an appropriate class of submanifolds to study infinite dimensional algebraic topology, especially infinite dimensional Morse theory. The proper-Fredholm condition assures, that Morse functions have only critical points of finite index. As the infinite dimensional sphere is contractible~\cite{Kuiper65}, this is a necessary condition for the development of Morse theory. The requirement for a submanifold to have finite codimension makes the normal bundle into a finite dimensional vector bundle. Both conditions together are required in the concept of proper-Fredholm submanifolds as defined by Chuu-Lian Terng~\cite{Terng89, PalaisTerng88}. Let us recall her definitions:

Let $V$ be a Hilbert space and $M\subset V$ an immersed submanifold with finite codimension, $TM$ its tangent bundle and $\nu(M)$ its normal bundle. Denote by $N_{\leq r}M$ the disc-normal bundle of $M$ consisting of discs of radius $r$. Define the \emph{endpoint map} by $Y:\nu(M)\rightarrow V, x\mapsto x+v$ for $x\in M$ and $v\in \nu(M)_x$.

Using these notations proper Fredholm submanifolds are defined as follows~\cite{Terng89, PalaisTerng88}:

\begin{definition}[proper-Fredholm submanifolds]
\label{definition:proper-Fredholm-submanifold}
An immersed submanifold $M\subset V$ of finite codimension is called \emph{proper Fredholm} (abbreviated PF) if the restriction of the endpoint map $Y:\nu(M)\longrightarrow N_{\leq r}M$ is proper Fredholm. 
\end{definition}

\begin{definition}
\label{isoparametric-pf-manifold}
 An immersed PF-submanifold $f:M\longrightarrow V$ of a Hilbert space $V$ is called isoparametric if
\begin{enumerate}
 \item $codim(M)$ is finite.
 \item $\nu(M)$ is globally flat.
 \item for any parallel normal field $v$ on $M$, the shape operators $A_{v(x)}$ and $A_{v(y)}$ are orthogonally equivalent for all $x,y$ in $M$. 
\end{enumerate}
\end{definition}

This class of submanifolds in Hilbert spaces are the guiding examples for the introduction of tame Fr\'echet submanifolds of co-finite type.

\subsection{Tame Fr\'echet submanifolds of co-finite type and of co-Banach type}

Tame Fr\'echet submanifolds of co-finite type are central objects of affine Kac-Moody geometry. They are a special case of the more general tame Fr\'echet submanifolds of co-Banach type.

\begin{definition}[Tame Fr\'echet submanifold of co-Banach type]
\label{TamesubmanifoldofBanachtype}
Let $F$ be a tame Fr\'echet spaces and $N$ a tame $F$-manifold.  A subset $M\subset N$ is a tame Fr\'echet submanifold of co-Banach type in $F$ iff the following is true:
 There is a complemented subspace $B\subset F$ which is a Banach space with a complement $F_0$ (hence $F= B\times F_0$), such that
 for every point $m\in M$ there are open sets $U(m)\subset N$, $V(m)\subset F$ and a tame Fr\'echet chart $\varphi_m:U(m)\longrightarrow V(m)\subset F$ such that
$$\varphi_{M}(M\cap U(m))= F_0\cap V(m)\,.$$
\end{definition}

In important applications, for example in the context of isoparametric submanifolds, it is enough to consider finite dimensional Banach spaces. Let us thus introduce the following definition:

\begin{definition}[Tame Fr\'echet submanifold of co-finite type]
\label{Tamesubmanifoldoffinitetype}
A tame Fr\'echet submanifold of co-Banach type is called of co-finite type if there is some $n\in \mathbb{N}$ such that $B=\mathbb{R}^n$.
\end{definition}

To justify the name ``tame Fr\'echet submanifold'' in the above definitions we have to establish, that the subset is actually a submanifold. This is the purpose of the following lemma:

\begin{lemma}
A submanifold of co-Banach type is a tame Fr\'echet manifold.
\end{lemma}

\begin{proof}~
Let $N$ be a tame Fr\'echet manifold and $M\subset N$ a submanifold of co-Banach type. We endow $M$ with the induced topology.
To prove the lemma we construct an atlas for $M$. Hence we show that there are charts, such that each point is in at least one chart and such that the chart transition functions are smooth tame maps. To this end we define charts via the maps $\varphi_m: U(m)\longrightarrow V(m)$ as follows: For each point $m\in M$ we define a chart by $\varphi_{M,m}:M\cap U(m) \longrightarrow F_0\cap V(m)$. We have to establish that chart transition functions are tame functions. 

As $N$ is a tame Fr\'echet manifold, chart transition function $\varphi_{m'}\circ\varphi_{m}^{-1}$ for charts $\varphi_{m'}$ and $\varphi_{m'}$ are tame maps, whenever they are defined. Hence the restriction of the chart transition functions $\varphi_{m'}\circ\varphi_{m}^{-1}$ to $M$  are also tame Fr\'echet maps, whenever they are defined.
\end{proof}

In the following result we describe the tame Fr\'echet structure on a tame Fr\'echet submanifold of co-Banach type. The key observation is, that the tame structure on the ambient manifold actually induces a natural tame structure on the submanifold. A map from a tame Fr\'echet space into a tame subamnifold of co-Banach type is tame iff it is tame as a map into the ambient space. This result is an important ingredient in the proof that affine Kac-Moody groups are tame Fr\'echet Lie groups in~\cite{Freyn09}.

\begin{theorem}
\label{mapinfrechetsubmanifold}
Let $M\subset F$ be a tame Fr\'echet submanifold of co-Banach type and let $i: M\hookrightarrow F$ denote the inclusion map. Let $H$ be a tame Fr\'echet space. A map $\varphi_M:H \longrightarrow M$ is tame if it is tame as a map $\varphi_F:=i\circ \varphi_M:H\rightarrow F$.
\end{theorem}

\begin{proof}
Let $\varphi$ be tame as a map $\varphi_F:H\rightarrow F$. Then the concatenation $\varphi_i\circ \varphi$ is tame for any chart $\varphi$ of $M$. Thus $\varphi$ is tame as a map into $M$.   
\end{proof}

\section{An implicit function theorem for tame maps}
\label{sect:implicit_function_theorem}

The aim of this section is to prove that the inverse image of a ``regular'' point of a tame map $\varphi$ of a tame Fr\'echet space $F$ into a Banach space (resp. $\mathbb{R}^n,n\in \mathbb{N}$) is a tame Fr\'echet submanifold of co-Banach (resp. co-finite) type. Following the finite dimensional blueprint we prove this result as a consequence of an implicit function theorem.

Usually inverse function theorems and implicit function theorems come in pairs.
In the literature various versions of implicit function theorems for classes of Fr\'echet spaces admitting smoothing operators (i.e.\ tame spaces) are proposed. See for example~\cite{Schwartz60, Sergeraert73, LojasiewiczZehnder79, Altmann84, SaintRaymond89, Poppenberg99, Krantz02, Ekeland11} and references therein. A nice implicit function theorem for maps from locally convex topological vector spaces to Fr\'echet spaces using metric estimates is described in~\cite{Glockner07}.

Richard Hamilton~\cite{Hamilton82} p.~212 proves the following theorem:

\begin{theorem}[Hamilton's implicit function theorem]
Let $F$, $G$ and $H$ be tame spaces and let $\Phi: U\subset F\times G\longrightarrow H$ be a smooth tame map.
Assume that whenever $\Phi(f,g)=0$, the partial derivative $D_f\Phi(f,g)$ is surjective, and there is a smooth tame map $V(f,g)h$ linear in $h$,
$$V:(U\subset F\times G)\times H \longrightarrow F\,,$$
and a smooth tame map $Q(f,g\{h,k\})$, bilinear in $h$ and $k$, such that for all $(f,g)$ in $U$ and all $h\in H$ we have: 
$$D_f\Phi(f,g)V(f,g)h=h+Q(f,g)\{\Phi(f,g),h\}\,,$$
so that $V$ is an approximate right inverse for $D_f\Phi$ with quadratic error $Q$. Then if $\Phi(f_0,g_0)=0$ for some $(f_0, g_0)\in U$ we can find neighborhoods of $f_0$ and $g_0$ such that for all $g$ in the neighborhood of $g_0$ there exists an $f$ in the neighborhood of $f_0$ with $\Phi(f,g)=0$. Moreover the solution $f=\Psi(g)$ is defined by a smooth tame map $\Psi$.
\end{theorem}

For details and a proof we refer to~\cite{Hamilton82}.
 
For the application we have in mind a far easier theorem is sufficient: 

\begin{theorem}[Implicit function theorem for tame maps]

\label{ImplicitefunctiontheoremfortameFrechetmaps}
Let $F$ be a tame Fr\'echet space, $V\cong W$ a Banach space. Let furthermore $\varphi:U\subset\ F\times V \longrightarrow W$ be a smooth tame map, such that the partial derivative $\frac{\partial}{\partial (y)}\varphi(z)$ is invertible at $z_0=(x_0,y_0)\in U$ (hence $\frac{\partial}{\partial (y)}\varphi(z)\in GL(W)$). Assume furthermore $\varphi(z_0)=0$.
Then there exist open subsets $U'\in F$ and $U''\in V$ and a smooth tame map $\psi:U'\longrightarrow U''$ such that
$$ \varphi(x,y)=0 \Leftrightarrow y=\psi(x)\,. $$
\end{theorem}

The proof of this result relies on the Nash-Moser inverse function theorem.

\begin{proof}~
The proof  proceeds in two steps: in the first step we reformulate the setting to prepare the use of the Nash-Moser inverse function theorem, in the second step we apply this theorem to get an inverse function, which by restriction yields the implicit function.

\begin{enumerate}
\item {\bf Prepare use of the Nash-Moser inverse function theorem.} 
Define the map $\Phi:U\subset F\times V \longrightarrow F\times W$ by $\Phi(x,y):=(x, \varphi(x,y))$.
As Cartesian products of tame spaces are tame (lemma~\ref{constructionoftamespaces}), $F\times W$ is a tame space; furthermore, as the identity map and $\varphi(x,y)$ are (smooth) tame maps, $\Phi$ is a (smooth) tame map (lemma~\ref{constructionoftamemaps}). 

The differential $D\Phi:TU\longrightarrow T(F\times W)$ of $\Phi$ is given by the formula
$$D\Phi(x,y)(h',h'')=\left(h', \frac{\partial}{\partial (x)}\varphi(x,y)h'+\frac{\partial}{\partial (y)}\varphi(x,y)h''\right)$$
for $(x,y)\in U$ and $(h',h'')\in T_{(x,y)}U$. As $\Phi$ is smooth also $D\Phi(x,y)(h',h'')$ is. Especially $\frac{\partial}{\partial (y)}\varphi(x,y)$ and $\frac{\partial}{\partial (y)}\varphi(z)^{-1}$ are smooth.

The general linear group $GL(W)$ of the Banach space $W$ is open in the space of all continuous linear endomorphisms of $V$. Hence there is some open neighborhood $\widetilde{U}\subset GL(W)$ of $\frac{\partial}{\partial (y)}\varphi(z)$. As $\varphi$ is continuous, there is some neighborhood $U'(z_0)$ such that for all $(x,y)\in U'$ the differential $\frac{\partial}{\partial (y)}\varphi(z)$ takes values in $\widetilde{U}\subset GL(W)$.  
For any $(x,y)\in U'$ and for every $k=(k',k'')\in F\times W$ the equation $D\Phi(x,y)(h',h'')=(k',k'')$ has a unique solution  $$(h',h'')=V\Phi(x,y)(k',k'')\in F\times V\,.$$  

The map $V\Phi(x,y)(k',k'')$ is of a similar structure as the map $D\Phi$:
$$V\Phi(x,y)(k',k'')=\left(k', -\frac{\partial}{\partial (y)}\varphi(z)^{-1}\frac{\partial}{\partial (x)}\varphi(x,y) k' +\frac{\partial}{\partial (y)}\varphi(z)^{-1}k''\right)\,.$$ 
As a combination of smooth tame maps, it is a smooth tame family of inverses.
In consequence, the map $\Phi$ satisfies the assumptions of the Nash-Moser inverse function theorem, which we will apply in the next step.

\item {\bf Apply the Nash-Moser inverse function theorem.}
Application of the Nash-Moser inverse function theorem yields open neighborhoods $U_0\subset U$ and $\widetilde{U}_0\subset F\times W$ such that $z_0\in U_0$ and $\Phi(z_0)=(x_0,0) \subset \widetilde{U}_0$ together with an inverse $\Phi^{-1}: \widetilde{U}_0 \longrightarrow U_0$ that is a smooth tame map.

Without loss of generality assume $\widetilde{U}_0=\widetilde{U}_0'\times \widetilde{U}_0''$ such that $\widetilde{U}_0''=\widetilde{U}_0\cap W$ and $\widetilde{U}_0'=\widetilde{U}_0\cap F$. Thus we can put $\Phi^{-1}(w_1, w_2)=(w_1, \psi(w_1, w_2))$ for some map $\psi: F\times W \rightarrow V$. The map $\psi$ is tame as $V$ is a Banach space.
Thus $$\Phi(x,y)=0 \Leftrightarrow y=\psi(x)\,.$$
This completes the proof.\qedhere
\end{enumerate}
\end{proof}

Remark that in contrast to the Nash-Moser inverse function theorem, it is enough, to assume invertibility of the linearization in a single point.

\section{Constructing tame Fr\'echet submanifolds}
\label{sect:construction_submanifolds}

In this final section we establish a characterization of tame Fr\'echet submanifolds of co-Banach type as the inverse image of regular points and describe some examples. We use the following definitions:

\begin{definition}[tame Fr\'echet regular point]
Let $F$ be a tame Fr\'echet space, $B$ a Banach space and $\varphi: U\subset F\rightarrow B$ a tame Fr\'echet map. A point $p\in F$ is called a \emph{tame Fr\'echet regular point of $\varphi$} iff the following is true:
\begin{enumerate}
 \item the differential $D\varphi(p)$ is in $GL(B)$,
 \item $\textrm{kern}(D\varphi(p))$ is a complemented subspace in $T_{p}F\cong F$.
\end{enumerate}
\end{definition}

\begin{definition}[tame Fr\'echet regular value]
Let $F$ be a tame Fr\'echet space, $B$ a Banach space and $\varphi: U\subset F\rightarrow B$ a tame Fr\'echet map. A point $g\in B$ is called a \emph{tame Fr\'echet regular value of $\varphi$} iff all points $p\in \varphi^{-1}(g)$ are tame regular points.
\end{definition}

 The main result is the following:

\begin{theorem}
\label{submanifoldsoffrechetspaceinmanifold}
Let $F$ be a tame Fr\'echet space and $B$ a Banach space. Let $N_{F}$ be a tame Fr\'echet $F$-manifold and $N_B$ a Banach $B$-manifold, such that the chart transition functions are quasi-isometries (hence $N_B$ is also a tame Fr\'echet manifold). Let furthermore $\varphi: N_{F} \rightarrow N_{B}$ be a tame Fr\'echet map and assume $g\in N_{B}$ to be a regular value for $\varphi$.
Then $\varphi^{-1}(g)$ is a tame Fr\'echet submanifold of co-Banach type. 
\end{theorem}

As $N_F$ and $N_B$ are tame Fr\'echet manifolds (see example~\ref{example:BanachmanifoldistameFrechet}), using appropriate charts, this result follows from the following lemma:

\begin{lemma}
\label{submanifoldsoffrechetspace}
Let $F$ be a tame Fr\'echet space and $B$ a Banach space. Let furthermore $\varphi: F \rightarrow B$ be a tame map and assume $0\in B$ to be a regular value for $\varphi$.
Then $\varphi^{-1}(0)$ is a tame Fr\'echet submanifold of co-Banach type. 
\end{lemma}

\begin{corollary}
\label{submanifoldsoffrechetspace_cofinite}
Let $F$ be a tame Fr\'echet space. Let furthermore $\varphi: F \rightarrow \mathbb{R}^n$ for some $n\in \mathbb{N}$ be a tame map and assume $g\in B$ to be a regular value for $\varphi$. Then $\varphi^{-1}(g)$ is a tame Fr\'echet submanifold of co-finite type. 
\end{corollary}

\begin{proof}
The proof of lemma~\ref{submanifoldsoffrechetspace} follows the finite dimensional blueprint:
\begin{enumerate} 
\item  {\bf Prepare use of theorem~\ref{ImplicitefunctiontheoremfortameFrechetmaps}} 
Let $a\in \varphi^{-1}(0)$. By assumption $0$ is a regular value. Hence $a$ is a regular point. Thus $F_0:=\textrm{kern}(d\varphi)_p$ is a complemented subspace of $F$. Hence we have a decomposition $F=F_0\times \overline{B}$. Furthermore $d\varphi_p\in GL(B)$. In consequence the map
$$\Phi: F_0\times \overline{B}\longrightarrow F_0\times B, (x,y)\mapsto (x, \varphi(x,y))$$
satisfies the assumptions of theorem~\ref{ImplicitefunctiontheoremfortameFrechetmaps}

\item {Apply the implicit function theorem}
Hence there are open subsets  $U=U'\times U''\subset F_0\times \overline{B}$ containing $0$ and $V=V'\times V''\in F_0\times B $ such that 
the map $\Phi: U\longrightarrow V$ is a tame Fr\'echet isomorphism. Clearly
$\Phi: U\cap {\varphi^{-1}(F_0}\longrightarrow V\cap F_0$ is an isomorphism.
Hence $\varphi^{-1}(0)$ is a tame Fr\'echet submanifold of co-Banach type as defined in~\ref{TamesubmanifoldofBanachtype}
\end{enumerate}
\end{proof}

\begin{example}[spheres in Fr\'echet spaces]
Let $F$ be a tame Fr\'echet space which is the inverse limit of Hilbert spaces. Hence all norms are deduced from metrics The $n$\ndash unit spheres $S_{F}^{n}:=\{x\in F| \|x\|_n=1 \}$ in $F$ are tame Fr\'echet submanifolds of co-finite type.  The same is true for intersections of finitely many $n$\ndash unit spheres.
\end{example}

In definition~\ref{isoparametric-pf-manifold} we recalled the notion of a PF-isoparametric submanifold in a Hilbert space. Let us extend this concept to certain Fr\'echet spaces as follows: 

\begin{definition}[isoparametric submanifold]
Let $F$ be a tame Fr\'echet space such that least one seminorm $\|\phantom{z}\|_k$ is induced by a metric $\langle\phantom{z},\phantom{z}\rangle_k$. Then we can defined the Hilbert space $H_k$ as the completion of $F$ with respect to the metric $\langle\phantom{z},\phantom{z}\rangle_k$. Then by definition  $F\subset H_k$. Let $M\subset H_k$ be a proper Fredholm isoparametric submanifold.  Then $M_F:=M\cap F\subset F$ is called a tame Fr\'echet isoparametric submanifold in $F$. 
\end{definition}

\begin{lemma}
$M_F$ is a tame Fr\'echet submanifold of co-finite type. 
\end{lemma}

\begin{proof}
Remark that by construction the normal bundle is flat and finite dimensional. More precisely, it has the same dimension $n$ as the normal bundle of the isoparametric hypersurface $H\subset H_k$. For any point $m\in M_F$ there exists an open neighborhood $U(m)$ and a chart $\phi_m: U(m)\rightarrow V(m)\subset F$. $\phi(U(m)\cap M_F)\subset V(m)$ has finite codimension $n$. We thus can find a complementary subspace $F_0\in F$ such that $F=F_0\times \mathbb{R}^n$ and a chart $\widetilde{\phi}: U(m)\rightarrow \widetilde{V}(m)$ such that $\phi(U(m)\cap M_F)=F_0\cap \widetilde{V}(m)$. 
\end{proof}

\def\cprime{$'$} \def\cprime{$'$}
\providecommand{\bysame}{\leavevmode\hbox to3em{\hrulefill}\thinspace}
\providecommand{\MR}{\relax\ifhmode\unskip\space\fi MR }
\providecommand{\MRhref}[2]{%
  \href{http://www.ams.org/mathscinet-getitem?mr=#1}{#2}
}
\providecommand{\href}[2]{#2}


\begin{thebibliography}{10}

\bibitem{Altmann84}
Mieczyslaw Altman, \emph{An inverse function theorem for {F}r\'echet spaces},
  Computational mathematics ({W}arsaw, 1980), Banach Center Publ.~13, PWN,
  Warsaw, 1984, pp.~683--699.

\bibitem{AnNeeb09}
Jinpeng An and Karl-Hermann Neeb, An implicit function theorem for {B}anach
  spaces and some applications, \emph{Math. Z.} \textbf{262} (2009), 627--643.

\bibitem{Berenstein91}
Carlos~A. Berenstein and Roger Gay, \emph{Complex {V}ariables - {A}n
  {I}ntroduction}, Graduate Texts in Mathematics 125, Springer Verlag, New
  York, 1991.

\bibitem{BertramGlocknerNeeb04}
W.~Bertram, H.~Gl{\"o}ckner and K.-H. Neeb, Differential calculus over general
  base fields and rings, \emph{Expo. Math.} \textbf{22} (2004), 213--282.

\bibitem{Bertram08}
Wolfgang Bertram, Differential geometry, {L}ie groups and symmetric spaces over
  general base fields and rings, \emph{Mem. Amer. Math. Soc.} \textbf{192}
  (2008), x+202.

\bibitem{BroeckerJaenich82}
Theodor Br{\"o}cker and Klaus J{\"a}nich, \emph{Introduction to differential
  topology}, Cambridge University Press, Cambridge, 1982, Translated from the
  German by C. B. Thomas and M. J. Thomas.

\bibitem{Burago01}
Dimitri Burago, Yuri Burago and Sergei Ivanov, \emph{A {C}ourse in {M}etric
  {G}eometry}, Graduate {S}tudies in {M}athematics~33, AMS, Providence, 2001.

\bibitem{Dodson12}
C.~T.~J. Dodson, Some recent work in {F}r\'echet geometry, \emph{Balkan J.
  Geom. Appl.} \textbf{17} (2012), 6--21.

\bibitem{DodsonGalanis04}
C.~T.~J. Dodson and G.~N. Galanis, Second order tangent bundles of infinite
  dimensional manifolds, \emph{J. Geom. Phys.} \textbf{52} (2004), 127--136.

\bibitem{Duren00}
Peter~L. Duren, \emph{Theory of $\textrm{{H}}^p$-{S}paces}, Dover Publications,
  New York, 2000.

\bibitem{Ekeland11}
Ivar Ekeland, An inverse function theorem in {F}r\'echet spaces, \emph{Ann.
  Inst. H. Poincar\'e Anal. Non Lin\'eaire} \textbf{28} (2011), 91--105.

\bibitem{Freyn09}
Walter Freyn, \emph{{K}ac-{M}oody symmetric spaces and universal twin
  buildings}, Ph.D. thesis, {U}niversit\"at {A}ugsburg, 2009.

\bibitem{Freyn12e}
Walter Freyn, \emph{{G}eometry of {K}ac-{M}oody symmetric spaces}, in
  preparation, 2012.

\bibitem{Freyn12d}
Walter Freyn, \emph{{H}olomorphic completions of affine {K}ac-{M}oody groups},
  submitted, 2012.

\bibitem{Freyn10a}
Walter Freyn, \emph{{K}ac-{M}oody geometry}, {G}lobal {D}ifferential
  {G}eometry, Springer, Heidelberg, 2012, pp.~55--92.

\bibitem{Freyn12a}
Walter Freyn, \emph{{P}seudo-{Riemannian} affine {K}ac-{M}oody symmetric
  spaces}, 2012.

\bibitem{Glockner06}
Helge Gl{\"o}ckner, Implicit functions from topological vector spaces to
  {B}anach spaces, \emph{Israel J. Math.} \textbf{155} (2006), 205--252.

\bibitem{Glockner07}
{H}elge Gl{\"o}ckner, {I}mplicit {F}unctions from {T}opological {V}ector
  {S}paces to {F}r\'echet {S}paces in the {P}resence of {M}etric {E}stimates,
  \emph{arXiv:math.GM/0612673 v5} (2007).

\bibitem{HeintzeGorodski12}
Claudio Gorodski and Ernst Heintze, Homogeneous structures and rigidity of
  isoparametric submanifolds in {H}ilbert space, \emph{J. Fixed Point Theory
  Appl.} \textbf{11} (2012), 93--136.

\bibitem{Guest97}
Martin Guest, \emph{Harmonic {M}aps, {L}oop {G}roups and {I}ntegrable
  {S}ystems}, {L}ondon {M}athematical {S}ociety {S}tudent {T}exts~38, London
  mathematical society, London, 1997.

\bibitem{Hamilton82}
Richard~S. Hamilton, The inverse function theorem of {N}ash and {M}oser,
  \emph{Bull. Amer. Math. Soc. (N.S.)} \textbf{7} (1982), 65--222.

\bibitem{HKZ00}
Haakan Hedenmalm, Boris Korenblum and Zhu Kehe, \emph{Theory of {B}ergman
  Spaces}, Graduate {T}exts in {M}athematics 199, Springer Verlag, New York,
  2000.

\bibitem{Heintze06}
Ernst Heintze, Towards symmetric spaces of affine {K}ac-{M}oody type,
  \emph{Int. J. Geom. Methods Mod. Phys.} \textbf{3} (2006), 881--898.

\bibitem{HeintzeLiu}
Ernst Heintze and Xiaobo Liu, Homogeneity of infinite-dimensional isoparametric
  submanifolds, \emph{Ann. of Math. (2)} \textbf{149} (1999), 149--181.

\bibitem{HPTT}
Ernst Heintze, Richard~S. Palais, Chuu-Lian Terng and Gudlaugur Thorbergsson,
  \emph{Hyperpolar actions on symmetric spaces}, Conf. Proc. Lecture Notes
  Geom. Topology, IV, pp.~214--245, Int. Press, Cambridge, MA, 1995.

\bibitem{Kac90}
Victor~G. Kac, \emph{Infinite-dimensional {L}ie algebras}, third ed, Cambridge
  University Press, Cambridge, 1990.

\bibitem{Khesin09}
{B}oris Khesin and Robert {W}endt, \emph{The {G}eometry of
  {I}nfinite-dimensional {G}roups}, Springer Verlag, Berlin, 2009.

\bibitem{Kobayashi11}
Shimpei Kobayashi, Real forms of complex surfaces of constant mean curvature,
  \emph{Trans. Amer. Math. Soc.} \textbf{363} (2011), 1765--1788.

\bibitem{Krantz02}
Steven~G. Krantz and Harold~R. Parks, \emph{The implicit function theorem},
  Birkh\"auser, Base, 2002.

\bibitem{Kriegl97}
Andreas Kriegl and Peter~W. Michor, \emph{The convenient setting of global
  analysis}, Mathematical Surveys and Monographs~53, American Mathematical
  Society, Providence, RI, 1997.

\bibitem{Kuiper65}
Nicolaas~H. Kuiper, The homotopy type of the unitary group of {H}ilbert space,
  \emph{Topology} \textbf{3} (1965), 19--30.

\bibitem{Kuiper71}
Nicolaas~H. Kuiper, \emph{Variet\'es hilbertiennes: aspects g\'eom\'etriques},
  Les Presses de l'Universit\'e de Montr\'eal, Montreal, Que., 1971, Suivi de
  deux textes de Dan Burghelea, S{\'e}minaire de Math{\'e}matiques
  Sup{\'e}rieures, No. 38 ({\'E}t{\'e}, 1969).

\bibitem{LojasiewiczZehnder79}
S.~{\L}ojasiewicz, Jr. and E.~Zehnder, An inverse function theorem in
  {F}r\'echet-spaces, \emph{J. Funct. Anal.} \textbf{33} (1979), 165--174.

\bibitem{Moody95}
Robert~V. Moody and Arturo Pianzola, \emph{Lie algebras with triangular
  decomposition}, {J}ohn {W}iley {S}ons, New York, 1995.

\bibitem{Muller06}
Olaf M{\"u}ller, \emph{{B}ounded {F}r\'echet geometry},
  ar{X}iv.org:math/0612379, 2006.

\bibitem{Nash56}
John Nash, The imbedding problem for {R}iemannian manifolds, \emph{Ann. of
  Math. (2)} \textbf{63} (1956), 20--63.

\bibitem{Neeb06}
Karl-Hermann Neeb, Towards a {L}ie theory of locally convex groups, \emph{Jpn.
  J. Math.} \textbf{1} (2006), 291--468.

\bibitem{Neeb07}
Karl-Hermann Neeb and Friedrich Wagemann, Lie group structures on groups of
  smooth and holomorphic maps on non-compact manifolds, \emph{preprint} (2007),
  39.

\bibitem{PalaisTerng88}
Richard Palais and Chuu-Lian Terng, \emph{Critical {P}oint {T}heory and
  {S}ubmanifold {G}eometry}, Lecture Notes in Mathematics 1353, Springer
  Verlag, Berlin, 1988.

\bibitem{Payne91}
Kevin~R. Payne, Smooth tame {F}r\'echet algebras and {L}ie groups of
  pseudodifferential operators, \emph{Comm. Pure Appl. Math.} \textbf{44}
  (1991), 309--337.

\bibitem{Payne89}
Kevin~Ray Payne, \emph{Smooth tame {F}rechet algebras and lie groups of
  pseudodifferential operators}, ProQuest LLC, Ann Arbor, MI, 1989, Thesis
  (Ph.D.)--State University of New York at Stony Brook.

\bibitem{Popescu05}
Bogdan Popescu, \emph{Infinite dimensional symmetric spaces}, Thesis,
  University of Augsburg, Augsburg, 2005.

\bibitem{Poppenberg99}
Markus Poppenberg, An application of the {N}ash-{M}oser theorem to ordinary
  differential equations in {F}r\'echet spaces, \emph{Studia Mathematica}
  \textbf{137} (2002), 101--121.

\bibitem{PressleySegal86}
Andrew Pressley and Graeme Segal, \emph{Loop groups}, Oxford Mathematical
  Monographs, The Clarendon Press Oxford University Press, New York, 1986,
  Oxford Science Publications.

\bibitem{SaintRaymond89}
Xavier Saint~Raymond, A simple {N}ash-{M}oser implicit function theorem,
  \emph{Enseign. Math. (2)} \textbf{35} (1989), 217--226.

\bibitem{Schaefer80}
Helmut~H. Schaefer, \emph{Topological vector spaces. 4th corr. printing.},
  Graduate Texts in Mathematics, Springer Verlag, New York, Heidelberg, Berlin,
  1980.

\bibitem{Schwartz60}
J.~Schwartz, On {N}ash's implicit functional theorem, \emph{Comm. Pure Appl.
  Math.} \textbf{13} (1960), 509--530.

\bibitem{SegalWilson85}
Graeme Segal and George Wilson, {L}oop groups and equations of {K}d{V} type,
  \emph{{P}ublications {M}ath\'ematiques de l'{IH\'ES}} \textbf{61} (1985),
  5--65.

\bibitem{Sergeraert73}
Francis Sergeraert, Un th\'eor\`eme de fonctions implicites. {A}pplications,
  \emph{Ann. Inst. Fourier} \textbf{23} (1973), 151--157.

\bibitem{Terng89}
Chuu-Lian Terng, Proper {F}redholm submanifolds of {H}ilbert space, \emph{J.
  Differential Geom.} \textbf{29} (1989), 9--47.

\bibitem{Tsvelik95}
Alexander Tsvelik, \emph{Quantum Field Theory in Condended Matter Physics},
  Cambridge University Press, Cambridge, 2003.

\bibitem{Zarnadze92}
D.~N. Zarnadze, Some topological and geometric properties of
  {F}r\'echet-{H}ilbert spaces, \emph{Izv. Ross. Akad. Nauk Ser. Mat.}
  \textbf{56} (1992), 1001--1020.

\end{thebibliography}
\end{document}